\documentclass[11pt,a4paper]{article}
\usepackage[T1]{fontenc}
\usepackage[utf8]{inputenc}
\usepackage{lmodern}
\usepackage[margin=28mm,headheight=14pt]{geometry}
\usepackage{amsmath,amssymb,amsthm,mathtools}
\usepackage{microtype}
\usepackage{enumitem}
\usepackage{xcolor}
\usepackage{fancyhdr}
\usepackage{hyperref}
\hypersetup{colorlinks=true,linkcolor=blue!45!black,citecolor=blue!45!black,
  urlcolor=blue!45!black,pdftitle={Operator-norm Sudakov minoration for Gaussian chaos of order two},
  pdfsubject={A proof of the endpoint in Talagrand's Research Problem 15.1.16}}
\setlist[enumerate]{itemsep=3pt,topsep=5pt}
\setlist[itemize]{itemsep=3pt,topsep=5pt}
\allowdisplaybreaks[2]
\numberwithin{equation}{section}
\newtheorem{theorem}{Theorem}[section]
\newtheorem{lemma}[theorem]{Lemma}
\newtheorem{proposition}[theorem]{Proposition}
\newtheorem{corollary}[theorem]{Corollary}
\theoremstyle{definition}

\theoremstyle{remark}
\newtheorem{remark}[theorem]{Remark}
\newcommand{\E}{\mathbb E}
\newcommand{\Prob}{\mathbb P}
\newcommand{\R}{\mathbb R}
\newcommand{\Id}{\mathrm{Id}}
\newcommand{\HS}{\mathrm{HS}}
\newcommand{\op}{\mathrm{op}}
\newcommand{\W}{\mathcal W}
\newcommand{\cE}{\mathcal E}

\newcommand{\cH}{\mathcal H}
\newcommand{\conv}{\operatorname{conv}}
\newcommand{\tr}{\operatorname{tr}}
\newcommand{\rank}{\operatorname{rank}}
\newcommand{\dist}{\operatorname{dist}}
\newcommand{\sym}{\operatorname{sym}}
\newcommand{\Pack}{P}
\newcommand{\cPack}{\widehat P}
\newcommand{\KL}{D}
\newcommand{\MI}{\mathsf I}
\newcommand{\Ent}{\mathsf H}
\newcommand{\law}{\mathcal L}
\newcommand{\ip}[2]{\langle #1,#2\rangle}
\newcommand{\norm}[1]{\lVert #1\rVert}

\title{\textbf{Operator-norm Sudakov minoration\\for Gaussian chaos of order two}\footnote{The authors were supported by National Science Centre, Poland grants 2022/47/B/ST1/02114 (W. Bednorz) and 2021/40/C/ST1/00330 (R. Meller). }.}
\author{Witold Bednorz\footnote{Institute of Mathematics, University of Warsaw, Banacha 2, 02-097 Warsaw, Poland.},  Rafa{\l} Martynek\footnote{Institute of Mathematics, University of Warsaw, Banacha 2, 02-097 Warsaw, Poland.} and Rafa{\l} Meller\footnote{Institute of Mathematics, University of Warsaw, Banacha 2, 02-097 Warsaw, Poland.}}
\date{Research manuscript\quad\textbullet\quad 20 September 2026}

\begin{document}
\maketitle
\begin{center}\textbf{AI (GPT-6 Astra) was used in this research.}\end{center}
\vspace{-1.5em}
\begin{abstract}
We prove that an operator-norm separated family of matrices satisfies
\[
\E\sup_{A\in T}g^TAh\ge c a\sqrt{\log|T|},
\]
where $g,h$ are independent
standard Gaussian vectors and $a$ is the separation. The main information
estimate concerns arbitrary separated coisometries: conditional entropy is
bounded by a source-dependent operator energy times $\log|T|$, up to an
additive quadratic term in the common row dimension. An adaptive Gaussian
experiment proves this estimate by charging actual information increments
to one weighted posterior-entropy potential. Convex separation and a
Gaussian covering estimate then yield a bounded-radius result. To reach
the general case, we first choose an operator scale preserving the Sudakov
ratio, apply the known Hilbert--Schmidt minoration, and recompute a common
Gaussian block compression at the retained entropy. This ordering preserves
the normalization needed by the coisometry argument.
\end{abstract}

\section{Statement, notation, and the structure of the proof}
\label{sec:statement}

All matrices are real. If $T\subset\R^{d\times n}$ is finite, let
\begin{equation}\label{eq:widths}
 S(T)=\E\sup_{A\in T}g^TAh,
 \qquad
 \W(T)=\E\sup_{A\in T}|g^TAh|,
\end{equation}
where $g\in\R^d$ and $h\in\R^n$ are independent standard Gaussian vectors.
The matrix inner product is $\ip AX=\tr(AX^T)$. We write
$\norm{A}_{\op}$, $\norm{A}_{\HS}$, and $\norm{A}_*$ for the operator,
Hilbert--Schmidt, and nuclear norms, respectively. All logarithms are
natural. Positive numerical constants may change from line to line.

\begin{theorem}\label{thm:main}
There is a numerical constant $c>0$ such that, for every finite family
$T=\{A_1,\ldots,A_m\}\subset\R^{d\times n}$ and every $a>0$,
\begin{equation}\label{eq:main-packing}
 \norm{A_i-A_j}_{\op}\ge a\quad(i\ne j)
 \quad\Longrightarrow\quad
 S(T)\ge c a\sqrt{\log m}.
\end{equation}
Consequently, with $N$ denoting the usual covering number,
\begin{equation}\label{eq:main-covering}
 \varepsilon\sqrt{\log N(T,\norm{\cdot}_{\op},\varepsilon)}
 \le C S(T)\qquad(\varepsilon>0).
\end{equation}
\end{theorem}

The covering estimate is the endpoint asked for in Research Problem
15.1.16 of \cite{TalagrandBook}. The existing Hilbert--Schmidt estimate
from Proposition 15.1.15 of that book, originating in
\cite{TalagrandChaos}, is an input to the proof; we state its exact form
in Section~\ref{sec:localization}.

For a subset $E$ of a normed space, $\Pack(E,a)$ is the largest cardinality
of a subset with pairwise distances at least $a$. The convex separation
number $\cPack(E,a)$ is the largest length of a sequence $x_1,\ldots,x_N$
in $E$ such that
\begin{equation}\label{eq:convex-packing-definition}
 \dist(x_j,\conv\{x_i:i<j\})\ge a\qquad(j\ge2).
\end{equation}
Convex separation is the notion introduced in \cite{AMST}. Unless a
different norm is specified, both packing notions below use the operator
norm.

The proof has four parts. Section~\ref{sec:information} proves the
coisometry entropy inequality. Section~\ref{sec:pruning} converts it into
convex separation at a fixed radius-to-separation ratio.
Section~\ref{sec:covering} bounds convex separation by bilinear width.
Sections~\ref{sec:compression}--\ref{sec:localization} produce a bounded
family in one common row space and complete the argument.

We will repeatedly use three elementary facts. Translation by a fixed
matrix preserves $S$. If $0\in T$, symmetry of the process gives
\begin{equation}\label{eq:absolute-comparison}
 S(T)\le\W(T)\le2S(T).
\end{equation}
Also, orthogonal compression on either side contracts $\W$, by conditioning
on the Gaussian coordinates that are removed. Finally, for every matrix $A$,
\begin{equation}\label{eq:one-matrix-lower}
 \E|g^TAh|=\sqrt{\frac2\pi}\,\E\norm{Ah}_2
 \ge\frac2\pi\norm{A}_{\op}.
\end{equation}
It follows that
\begin{equation}\label{eq:two-point}
 S(\{A,B\})=\frac12\E|g^T(A-B)h|
 \ge\frac1\pi\norm{A-B}_{\op}.
\end{equation}

\section{A weighted entropy inequality for coisometries}
\label{sec:information}

Fix $0<\delta\le2$ and a family $U_1,\ldots,U_m\in\R^{r\times n}$,
where $r\ge1$, satisfying
\begin{equation}\label{eq:coisometries}
 U_iU_i^T=\Id_r,
 \qquad \norm{U_i-U_j}_{\op}\ge\delta\quad(i\ne j).
\end{equation}
Let $(I,J)$ be any pair of indices taking values in $\{1,\ldots,m\}$,
and write $\mu_i=\Prob(I=i)$. Define, for $\mu_i>0$,
\begin{equation}\label{eq:lambda}
 \Lambda_i=\frac12\E\big[(U_i-U_J)(U_i-U_J)^T\mid I=i\big],
 \qquad \lambda_i=\norm{\Lambda_i}_{\op}.
\end{equation}
Terms with $\mu_i=0$ are omitted. Notice that $0\le\lambda_i\le2$.
Entropy and mutual information, denoted by $\Ent$ and $\MI$, use natural
logarithms; information involving continuous observations is defined by
relative entropy, not differential entropy.

\begin{theorem}\label{thm:entropy}
Under \eqref{eq:coisometries},
\begin{equation}\label{eq:weighted-entropy-theorem}
 \Ent(I\mid J)
 \le C_\delta r^2+
 C_\delta\sum_i\mu_i\lambda_i\log\frac1{\mu_i}.
\end{equation}
If $I$ and $J$ have the same marginal distribution, $p=\log m$, and
$\cE=\sum_i\mu_i\lambda_i$, then
\begin{equation}\label{eq:stationary-entropy}
 \Ent(J\mid I)=\Ent(I\mid J)
 \le C_\delta r^2+C_\delta p\cE.
\end{equation}
The constants depend only on $\delta$, and not on $n$, $m$, or the
distribution of $(I,J)$.
\end{theorem}

\subsection{A posterior identity and a Gaussian information bound}

Suppose the current posterior is $\nu$, the next observation has law
$R_i$ conditional on $I=i$, and $R=\sum_i\nu_iR_i$. For fixed nonnegative
weights $\lambda_i$, set
\[
 \phi(\nu)=\sum_i\nu_i\lambda_i\log\frac1{\nu_i}.
\]
Bayes' formula gives the exact identity
\begin{equation}\label{eq:weighted-drop}
 \phi(\nu)-\E\phi(\nu_{\mathrm{new}})
 =\sum_i\nu_i\lambda_i\KL(R_i\Vert R)\ge0.
\end{equation}
This also holds for deterministic finite-alphabet messages. All expectations
in this identity are conditional on the current history.

\begin{lemma}\label{lem:sourcewise-kl}
Let $a_i\in\R^n$ have norm one, let $\nu$ be a probability distribution,
and put $\bar a=\sum_i\nu_i a_i$ and $v=1-\norm{\bar a}_2^2$. For
$P_i=N(a_i,\Id_n)$ and $P=\sum_i\nu_iP_i$, if $v>\kappa>0$, then
\begin{equation}\label{eq:sourcewise-kl}
 \KL(P_i\Vert P)\ge\frac{\kappa^2}{16}
 \qquad\text{for every }i\text{ with }\nu_i>0.
\end{equation}
\end{lemma}
\begin{proof}
For every $z\in\R^n$, Hoeffding's inequality and the Gaussian
moment-generating function imply
\[
 \log\E_P\exp\ip z{Y-\bar a}\le\norm z_2^2.
\]
Indeed, $\ip z{a_i}$ lies in an interval of length $2\norm z_2$.
The variational inequality for relative entropy now gives
\[
 \KL(P_i\Vert P)
 \ge\sup_z\{\ip z{a_i-\bar a}-\norm z_2^2\}
 =\frac14\norm{a_i-\bar a}_2^2.
\]
On the other hand,
\[
 \norm{a_i-\bar a}_2\ge1-\norm{\bar a}_2
 =1-\sqrt{1-v}\ge v/2>\kappa/2.
\]
\end{proof}

\subsection{The adaptive experiment}

We construct a transcript from $I$ and independent Gaussian noises. The
construction does not use $J$. At any time, let $\nu$ be the posterior of
$I$ given the public transcript, and put $M=\sum_i\nu_iU_i$.
Fix a $\delta/8$-net $\mathcal N\subset S^{r-1}$ with
\begin{equation}\label{eq:direction-alphabet}
 |\mathcal N|\le(1+16/\delta)^r.
\end{equation}
Set
\[
 b=\delta/4,\qquad \kappa=b^2/2=\delta^2/32.
\]

Each round consists of the following operations.
\begin{enumerate}
\item There is at most one index $i_0$ with
      $\norm{U_{i_0}-M}_{\op}<\delta/2$. If $I=i_0$, send a terminal
      symbol. The decoder identifies $I=i_0$, and the experiment ends.
\item Otherwise send a direction $x\in\mathcal N$ for which
      \begin{equation}\label{eq:direction-witness}
      \norm{(U_I-M)^Tx}_2\ge b.
      \end{equation}
      Use, for example, the first such direction in a fixed ordering.
\item For this fixed $x$, take successive observations
      \begin{equation}\label{eq:gaussian-query}
       Y=U_I^Tx+G,\qquad G\sim N(0,\Id_n),
      \end{equation}
      with fresh independent noise, while the current posterior row variance
      \[
       v=\sum_i\nu_i\norm{U_i^Tx-M^Tx}_2^2
        =1-\norm{M^Tx}_2^2
      \]
      is greater than $\kappa$. Once $v\le\kappa$, begin a new round.
\end{enumerate}
The direction in step 2 exists because $\norm{U_I-M}_{\op}\ge\delta/2$
and $\norm{U_I-M}_{\op}\le2$. Approximating a maximizing left direction
by the net loses at most $\delta/4$. If the variance is already at most
$\kappa$ after the direction message, the round has no Gaussian query.
Continuation and stopping of a query block are determined by its public
history. The source-dependent choices are precisely the finite-alphabet
messages in steps 1 and 2.

\subsection{Termination and the cost of direction messages}

At every active Gaussian query, Lemma~\ref{lem:sourcewise-kl} applies with
$a_i=U_i^Tx$ and $c_0=\kappa^2/16$. Thus the ordinary conditional mutual
information of a query is at least $c_0$. Applying the chain rule first to
finite chronological transcript truncations yields
\begin{equation}\label{eq:query-count}
 c_0\E N_{\mathrm G}\le\Ent(I).
\end{equation}
In particular, every Gaussian block ends almost surely. This conclusion is
obtained before assuming that the full experiment terminates.

Consider one nonterminal round. Write $M_-$ and $M_+$ for the posterior
matrix means before the direction message and at the end of the round.
Conditional on the preceding history and the message $x$, every compatible
index satisfies \eqref{eq:direction-witness}. Conditional variance
decomposition at the end of the block gives, writing $\cH_-$ for the
preceding history,
\begin{align}
 &\E\big[\norm{(M_+-M_-)^Tx}_2^2\mid\cH_-,x\big]\notag\\
 &\qquad=\E\big[\norm{(U_I-M_-)^Tx}_2^2\mid\cH_-,x\big]\notag\\
 &\qquad\quad-\E\big[\norm{(U_I-M_+)^Tx}_2^2\mid\cH_-,x\big]
 \ge b^2-\kappa=\kappa.\label{eq:round-progress}
\end{align}
The matrix posterior means form an $L^2$-bounded martingale. Summing its
squared Hilbert--Schmidt increments over complete rounds, and padding by
constant values after termination, gives
\[
 \sum_{\text{rounds}}\E\norm{M_+-M_-}_{\HS}^2
 \le\E\norm{U_I}_{\HS}^2=r.
\]
The increment here includes the direction message itself. Since
$\norm{(M_+-M_-)^Tx}_2\le\norm{M_+-M_-}_{\HS}$, the expected number of
nonterminal rounds is at most $r/\kappa$. Infinite rounds therefore have
probability zero, and the terminal message eventually identifies $I$.

Each message has at most $1+|\mathcal N|$ possible values. Its conditional
mutual information, also when conditioning on $J$, is at most the logarithm
of this number. Consequently the total information of all such messages is
at most
\begin{equation}\label{eq:message-cost}
 (1+r/\kappa)\log(1+|\mathcal N|)\le C_\delta r^2.
\end{equation}
This is a chronological chain-rule calculation. No additional description
of the number of queries or rounds is needed: the next operation is
determined by the preceding public transcript.

\subsection{Charging partial information to one potential}

At a Gaussian query with current history $\cH=h$, compare its output with
the reference law $N(U_J^Tx,\Id_n)$. The Gaussian relative-entropy identity
and the mixture decomposition of relative entropy give
\begin{equation}\label{eq:conditional-query-info}
 \MI(I;Y\mid J,\cH=h)
 \le\frac12\E[\norm{(U_I-U_J)^Tx}_2^2\mid\cH=h].
\end{equation}
Since the transcript is generated without $J$,
\begin{equation}\label{eq:conditional-independence}
 \law(J\mid I,\cH)=\law(J\mid I).
\end{equation}
The right side of \eqref{eq:conditional-query-info} is therefore
\[
 \sum_i\nu_i x^T\Lambda_i x\le\sum_i\nu_i\lambda_i.
\]
By \eqref{eq:weighted-drop} and Lemma~\ref{lem:sourcewise-kl}, the expected
decrease of $\phi$ at this same query is at least
$c_0\sum_i\nu_i\lambda_i$. Hence
\begin{equation}\label{eq:local-information-charge}
 \MI(I;Y\mid J,\cH=h)
 \le c_0^{-1}\big(\phi(\nu)-\E[\phi(\nu_{\mathrm{new}})\mid\cH=h]\big).
\end{equation}
Every other observation also has nonnegative expected potential decrease.
Summing \eqref{eq:local-information-charge} over Gaussian queries thus
costs at most $c_0^{-1}\phi(\mu)$. Finite truncations justify the summation;
the conditional information of the full transcript is the increasing limit
of that of its finite prefixes. Since the transcript identifies $I$,
\eqref{eq:message-cost} proves \eqref{eq:weighted-entropy-theorem}.

For the stationary conclusion, put $p=\log m$. Since $\lambda_i\le2$,
\begin{align}
 \phi(\mu)
 &=p\sum_i\mu_i\lambda_i+
   \sum_i\mu_i\lambda_i\log\frac1{m\mu_i}\notag\\
 &\le p\cE+2/e.\label{eq:weighted-entropy-initial}
\end{align}
For the last bound, discard negative summands and use
$t\log(1/(mt))\le1/(em)$. Equality of the marginals gives
$\Ent(I\mid J)=\Ent(J\mid I)$, and the constant is absorbed into $C_\delta r^2$.
This completes the proof of Theorem~\ref{thm:entropy}.

\begin{remark}
The charge in \eqref{eq:local-information-charge} is for the information
actually acquired at the current query. It is not a charge for a whole-label
decoding error. The sourcewise lower bound in
Lemma~\ref{lem:sourcewise-kl} is what permits the fixed weights $\lambda_i$
to remain inside the telescoping sum.
\end{remark}

\section{From the entropy inequality to convex separation}
\label{sec:pruning}

\begin{proposition}\label{prop:coisometry-pruning}
For every fixed $\delta>0$, there are $b_\delta>0$ and $C_\delta<\infty$
such that every family satisfying \eqref{eq:coisometries} obeys
\begin{equation}\label{eq:coisometry-pruning}
 \log\cPack(\{U_i\},b_\delta)
 \ge\tfrac34\log m-C_\delta r^2.
\end{equation}
\end{proposition}
\begin{proof}
The assertion is immediate for $m=1$, so assume $p=\log m>0$. Set
\[
 D_{ij}=\tfrac12(U_i-U_j)(U_i-U_j)^T
       =\Id_r-\sym(U_iU_j^T).
\]
For positive semidefinite matrices $W_i$ with $\tr W_i=1$, and a parameter
$\beta>0$, consider the strictly positive matrix
\[
 K(W)_{ij}=\exp(-\beta\tr(W_iD_{ij})).
\]
If $Q=(Q_{ij})$ is a stationary joint distribution, with equal row and
column marginal $\mu$, put
\[
 h(Q)=-\sum_{i,j}Q_{ij}\log(Q_{ij}/\mu_i),\qquad
 \cE(Q)=\sum_i\norm{\sum_jQ_{ij}D_{ij}}_{\op}.
\]
The Perron variational formula is
\begin{equation}\label{eq:perron-variational}
 \log\rho(K(W))=\max_{Q\ \mathrm{stationary}}
 \left\{h(Q)-\beta\sum_{i,j}Q_{ij}\tr(W_iD_{ij})\right\}.
\end{equation}
To verify it, let $v>0$ be a right Perron vector and set
$R_{ij}=K(W)_{ij}v_j/(\rho(K(W))v_i)$. This is a stochastic matrix.
For stationary $Q$, the expression on the right equals
\[
 \log\rho(K(W))-\sum_i\mu_i\KL(Q_{i\cdot}/\mu_i\Vert R_{i\cdot}).
\]
It is at most $\log\rho(K(W))$, with equality for a stationary Markov
chain having transition matrix $R$.

The function in braces in \eqref{eq:perron-variational} is concave and
continuous in $Q$ and affine in $W$ on compact convex sets. The finite
dimensional minimax theorem \cite{Sion} therefore gives
\begin{equation}\label{eq:perron-minimax}
 \min_W\log\rho(K(W))=\max_Q\{h(Q)-\beta\cE(Q)\}.
\end{equation}
Theorem~\ref{thm:entropy} supplies constants $A_\delta\ge1,B_\delta$ with
$h(Q)\le A_\delta p\cE(Q)+B_\delta r^2$. Choose
$\beta=A_\delta p$ and a minimizing $W$. Then
$\rho(K(W))\le\exp(B_\delta r^2)$.

Normalize a positive right Perron vector $\pi$ to have sum one, put
$b_\delta=(4A_\delta)^{-1}$, and define
\[
 C_i=\{j:\tr(W_iD_{ij})<b_\delta\}.
\]
Each $C_i$ contains $i$, and
\begin{equation}\label{eq:cap-mass}
 \pi(C_i)\le e^{\beta b_\delta}(K(W)\pi)_i
 =e^{\beta b_\delta}\rho(K(W))\pi_i.
\end{equation}
Take independent exponential random variables with rates $\pi_i$.
Retain $i$ when its clock is the smallest among the clocks indexed by $C_i$.
The expected number retained is
\[
 \sum_i\frac{\pi_i}{\pi(C_i)}
 \ge m e^{-\beta b_\delta}/\rho(K(W))
 \ge\exp(3p/4-B_\delta r^2).
\]
Choose a realization with at least this many retained indices and order
them by increasing clock values. For an earlier retained $j$ and a later
retained $i$, one has $j\notin C_i$. The matrices $Z_i=W_iU_i$ satisfy
\[
 \norm{Z_i}_*=\tr W_i=1,
 \qquad
 \ip{Z_i}{U_i-U_j}=\tr(W_iD_{ij})\ge b_\delta.
\]
They certify the required convex separation.
\end{proof}

\begin{proposition}\label{prop:bounded-pruning}
Fix $L\ge1$. If $E\subset\R^{r\times n}$ is a finite $a$-separated
family and $\sup_{A\in E}\norm{A}_{\op}\le La$, then, for
$K=\conv(E\cup-E)$,
\begin{equation}\label{eq:bounded-pruning}
 \log|E|\le C_Lr^2+C_L\log\cPack(K,c_La).
\end{equation}
\end{proposition}
\begin{proof}
Divide the matrices by $La$, so that they are contractions separated by
$1/L$. For each resulting matrix $A$, form
\[
 U_A=[\,A\ \ F_A\,],\qquad F_A=(\Id_r-AA^T)^{1/2}.
\]
These are coisometries and retain the separation. By
Proposition~\ref{prop:coisometry-pruning}, they contain a convexly
$b_{1/L}$-separated sequence of length at least
$\exp(\tfrac34\log|E|-C_Lr^2)$.

The completion matrices $F_A$ lie in the operator unit ball of the
$d=r(r+1)/2$ dimensional space of symmetric matrices. A volumetric cover
by balls of radius $b_{1/L}/4$ has at most
$(1+8/b_{1/L})^d$ members. Retain the largest corresponding subsequence.
Its completion matrices have diameter at most $b_{1/L}/2$.
For any convex combination of preceding points, the completion-coordinate
difference consequently has norm at most $b_{1/L}/2$. Since
\[
 \norm{[\,A\ \ F\,]}_{\op}\le\norm A_{\op}+\norm F_{\op},
\]
the original-coordinate sequence is convexly $b_{1/L}/2$-separated.
The entropy lost in selecting the bin is at most $C_Lr^2$.
Rescaling proves \eqref{eq:bounded-pruning}.
\end{proof}

\begin{remark}
Only convex separation is transferred through the coisometry lift.
No comparison between the bilinear widths before and after that nonlinear
lift is used. The fixed value of $L$ is essential for the uniform constants
in Proposition~\ref{prop:bounded-pruning}.
\end{remark}

\section{A convex-packing bound from bilinear width}
\label{sec:covering}

Let $K\subset\R^{r\times n}$ be compact, convex, and symmetric. Define
\[
 \alpha(X)=\sup_{A\in K}\ip AX,
 \qquad W=\E\alpha(gh^T),
 \qquad B_*=\{X:\norm X_*\le1\}.
\]
We may assume $W>0$: if $W=0$, then $K=\{0\}$ by
\eqref{eq:one-matrix-lower}, and the convex-packing conclusion is immediate.
The function $\alpha$ may be a seminorm. Its unit ball is denoted by
$K^\circ$; the covering arguments below remain valid for a seminorm.

\begin{lemma}[Gaussian cross terms]\label{lem:cross}
Let $V\in\R^{n\times k}$ satisfy $\norm V_{\HS}\le1$, and let $G$ be
an $r\times k$ standard Gaussian matrix. Then
\begin{equation}\label{eq:cross-bound}
 \E\alpha(GV^T)\le\sqrt3 W.
\end{equation}
Similarly, if $\norm U_{\HS}\le1$ and $H$ is an $n\times k$ standard
Gaussian matrix, then $\E\alpha(UH^T)\le\sqrt3 W$.
\end{lemma}
\begin{proof}
Assume $V\ne0$, let $h\sim N(0,\Id_n)$ be independent of $G$, and set
$z=V^Th/\norm{V^Th}_2$. Conditional on $h$, $Gz$ is a standard Gaussian
vector, so it is independent of $h$. With $D=\E(zh^T)$, Jensen's inequality
therefore gives $\E\alpha(GD)\le W$.

In singular coordinates for $V$, with nonzero singular values $\sigma_j$,
the corresponding singular values of $D$ are
\[
 d_j=\sigma_j\E\frac{h_j^2}
                    {(\sum_\ell\sigma_\ell^2h_\ell^2)^{1/2}}.
\]
Under the probability measure with density $h_j^2$, Jensen's inequality
and Cauchy--Schwarz show that
\[
 \E\frac{h_j^2}{(\sum_\ell\sigma_\ell^2h_\ell^2)^{1/2}}
 \ge\frac1{(\sum_\ell\sigma_\ell^2+2\sigma_j^2)^{1/2}}
 \ge\frac1{\sqrt3}.
\]
Thus $VV^T\preceq3D^TD$. Gaussian covariance domination, realized by
adding an independent centered Gaussian matrix, implies
$\E\alpha(GV^T)\le\sqrt3\E\alpha(GD)\le\sqrt3W$.
The second assertion follows by transposition.
\end{proof}

\begin{lemma}[Uniform Gaussian small balls]\label{lem:dual-cover}
For every $t>0$,
\begin{equation}\label{eq:dual-cover}
 \log N\left(B_*,CW(t^{-1/2}+rt^{-1})K^\circ\right)\le2t+\log4.
\end{equation}
\end{lemma}
\begin{proof}
Factor any $X\in B_*$ as $X=UV^T$, padding
with zero columns to use exactly $r$ columns, so that
$\norm U_{\HS},\norm V_{\HS}\le1$. Let $G\in\R^{r\times r}$ and
$H\in\R^{n\times r}$ be independent standard Gaussian matrices.
Lemma~\ref{lem:cross} and the triangle inequality give
\begin{align*}
 \E\alpha(\sigma^2GH^T+\sigma GV^T+\sigma UH^T)
 &\le \sigma^2 rW+2\sqrt3\sigma W\\
 &\le C W(\sigma+r\sigma^2).
\end{align*}
Let $u=C_0W(\sigma+r\sigma^2)$ with $C_0$ large enough. The event that
this seminorm is at most $u$ has Gaussian probability at least $1/2$.
Translate $(G,H)$ by $(U/\sigma,V/\sigma)$. Its image lies in
\[
 \{(G',H'):\alpha(\sigma^2G'H'^T-X)\le u\}.
\]
For standard Gaussian measure $\gamma$ and any measurable set $A$,
Cauchy--Schwarz applied to the Gaussian likelihood ratio gives
\[
 \gamma(A+z)\ge\gamma(A)^2e^{-\norm z_2^2}.
\]
The translation just used has squared norm at most $2/\sigma^2$.
Consequently, if $\nu_\sigma$ is the law of $\sigma^2GH^T$, then
\begin{equation}\label{eq:uniform-small-ball}
 \nu_\sigma(X+uK^\circ)\ge\tfrac14e^{-2/\sigma^2}
 \qquad(X\in B_*).
\end{equation}
A set of points of $B_*$ at pairwise $\alpha$-distance greater than $2u$
has disjoint $u$-balls, and so has cardinality at most $4e^{2/\sigma^2}$.
A maximal such set covers $B_*$ by $2u$-balls. Put $\sigma=t^{-1/2}$
and absorb the numerical factors into $C$.
\end{proof}

\begin{proposition}\label{prop:convex-upper}
For every $a>0$,
\begin{equation}\label{eq:convex-upper}
 \log\cPack(K,a)
 \le C\left[1+(W/a)^2+rW/a\right].
\end{equation}
\end{proposition}
\begin{proof}
Suppose $A_1,\ldots,A_N\in K$ are convexly $a$-separated. By
Hahn--Banach, there are $Y_j\in B_*$ such that
\[
 \ip{Y_j}{A_j-A_i}\ge a\qquad(i<j).
\]
For $j=1$ choose any $Y_1\in B_*$. Use Lemma~\ref{lem:dual-cover} with
\[
 t=C_0[1+(W/a)^2+rW/a]
\]
to cover $B_*$ by at most $\exp(Ct)$ balls of $\alpha$-radius $a/8$.
If $Y_i,Y_j$ lie in the same such ball and $i<j$, then
\[
 \ip{Y_j}{A_j}-\ip{Y_i}{A_i}
 =\ip{Y_j}{A_j-A_i}+\ip{Y_j-Y_i}{A_i}\ge3a/4.
\]
Write $D=\sup_{A\in K}\norm A_{\op}$. All the numbers
$\ip{Y_i}{A_i}$ lie in $[-D,D]$, so each ball contains at most
$1+8D/(3a)$ of the witnesses. By \eqref{eq:one-matrix-lower},
$D\le(\pi/2)W$. It follows that
\[
 \log N\le Ct+\log(1+CW/a)\le C[1+(W/a)^2+rW/a].
\]
\end{proof}

\begin{corollary}[Bounded-radius row estimate]\label{cor:bounded-row}
Fix $L\ge1$. If $E\subset\R^{r\times n}$ is $a$-separated and
$\sup_{A\in E}\norm A_{\op}\le La$, then
\begin{equation}\label{eq:bounded-row}
 \log|E|\le C_L\left[r^2+1+(\W(E)/a)^2+r\W(E)/a\right].
\end{equation}
In particular,
\begin{equation}\label{eq:bounded-row-sqrt}
 \sqrt{\log|E|}\le C_L(1+r+\W(E)/a).
\end{equation}
\end{corollary}
\begin{proof}
Apply Proposition~\ref{prop:bounded-pruning} and then
Proposition~\ref{prop:convex-upper} to $K=\conv(E\cup-E)$. Passing to
this symmetric convex hull preserves $\W$.
\end{proof}

\section{A common compression with controlled radius}
\label{sec:compression}

We first record the elementary Gaussian estimates needed below. If $G$ is
a $k\times d$ standard Gaussian matrix and $E\in\R^{d\times n}$, then
\begin{equation}\label{eq:gaussian-product-norm}
 \E\norm{GE}_{\op}\le\norm E_{\HS}+\sqrt{k}\norm E_{\op}.
\end{equation}
Indeed, compare the Gaussian process $u^TGEv$, indexed by unit $u,v$,
with $\norm E_{\op}\ip gu+\ip h{Ev}$. The latter has larger increment
variances: the difference is
$2(1-\ip u{u'})(\norm E_{\op}^2-\ip{Ev}{Ev'})\ge0$.
The Gaussian comparison inequality proves
\eqref{eq:gaussian-product-norm}.

If $P=k^{-1/2}G$ and $V\subset\R^d$ has dimension at most $d_0$, then,
for $k\ge C(d_0+1)$, with probability at least $0.999$,
\begin{equation}\label{eq:gaussian-subspace}
 \tfrac12\norm v_2\le\norm{Pv}_2\le2\norm v_2\qquad(v\in V).
\end{equation}
For completeness, take a $1/8$-net of the unit sphere of $V$, of size at
most $17^{d_0}$. The chi-square concentration bound gives norms in
$[3/4,5/4]$ on that net except with probability
$2\exp(Cd_0-ck)$. The net approximation yields
\eqref{eq:gaussian-subspace}.

\begin{lemma}\label{lem:compression}
Let $A_1,\ldots,A_m\in\R^{d\times n}$ be $a$-separated and satisfy
\[
 \sup_i\norm{A_i}_{\op}\le La,
 \qquad \sup_i\norm{A_i}_{\HS}\le H,
 \qquad L\ge1.
\]
There is a family $F\subset\R^{k\times N}$, for some $N$, such that
\begin{align}
 k&\le C(1+H^2/a^2),\label{eq:compression-dim}\\
 \log|F|&\ge\tfrac12\log(m/2),\label{eq:compression-card}\\
 \norm{B-B'}_{\op}&\ge ca\quad(B\ne B'\in F),\label{eq:compression-sep}\\
 \sup_{B\in F}\norm B_{\op}&\le CLa,
 \qquad \W(F)\le C\W(\{A_i\}).\label{eq:compression-radius-width}
\end{align}
In particular, the dimension $k$ is a common row dimension.
\end{lemma}
\begin{proof}
Choose a small numerical $\eta>0$, then a sufficiently small numerical
$\tau/a>0$, and finally a sufficiently large numerical constant in
\[
 k=\left\lceil C_0(1+H^2/a^2)\right\rceil.
\]
Split each matrix at singular-value threshold $\tau$ as $A_i=B_i+E_i$.
Then
\[
 \rank B_i\le H^2/\tau^2,
 \quad\norm{E_i}_{\op}\le\tau,
 \quad\norm{E_i}_{\HS}\le H,
 \quad\norm{B_i}_{\op}\le La.
\]
Let $P\in\R^{k\times d}$ and $Q\in\R^{n\times k}$ be independent,
with entries $N(0,1/k)$.

Call an index good if $P$ embeds the range of $B_i$, and $Q^T$ embeds its
row space, with the bounds \eqref{eq:gaussian-subspace}, and if
\begin{equation}\label{eq:tail-good}
 \max\{\norm{PE_i}_{\op},\norm{E_iQ}_{\op},
                   \norm{PE_iQ}_{\op}\}\le\eta a.
\end{equation}
For each individual index this event has probability at least $0.99$.
To see the tail assertion, \eqref{eq:gaussian-product-norm} gives
\[
 \E\norm{PE_i}_{\op},\ \E\norm{E_iQ}_{\op}
 \le\tau+H/\sqrt{k}.
\]
Conditioning on $Q$ and using
$\E\norm{E_iQ}_{\HS}\le\norm{E_i}_{\HS}$ also gives
\[
 \E\norm{PE_iQ}_{\op}\le\tau+2H/\sqrt{k}.
\]
Choose the constants so that Markov's inequality bounds the total tail
failure probability by $0.005$. Equation~\eqref{eq:gaussian-subspace}
controls the two remaining failures. All these choices are numerical.

Define the common linear map
\begin{equation}\label{eq:Phi}
 \Phi(A)=\begin{pmatrix}PAQ&PA\\AQ&0\end{pmatrix}.
\end{equation}
For a good index, the embedding upper bounds and
\eqref{eq:tail-good} imply
\begin{equation}\label{eq:Phi-radius}
 \norm{\Phi(A_i)}_{\op}\le8La+3\eta a\le CLa.
\end{equation}

For good indices $i,j$, choose a left inverse $L_i$ of $P$ on the range
of $B_i$, and a right inverse $R_j$ of $Q$ on the row space of $B_j$.
They have norms at most two and satisfy $L_iPB_i=B_i$ and $B_jQR_j=B_j$.
For $D=B_i-B_j$, the exact identity
\begin{equation}\label{eq:compression-identity}
 D=L_iPD+DQ R_j-L_iPDQ R_j
\end{equation}
gives
\[
 \norm D_{\op}\le2\norm{PD}_{\op}+2\norm{DQ}_{\op}
                                      +4\norm{PDQ}_{\op}.
\]
On the other hand $\norm D_{\op}\ge a-2\tau$. The tail bound for each
of $i,j$ shows, on choosing $\eta$ sufficiently small, that
\begin{equation}\label{eq:Phi-separation}
 \norm{\Phi(A_i)-\Phi(A_j)}_{\op}\ge c_0a.
\end{equation}
This is a simultaneous conclusion for every pair of good indices; it
does not require a union bound over pairs.

Let $T=\{A_i\}$. Gaussian conditioning gives
\[
 \E_P\W(PT)\le\W(T),\qquad
 \E_Q\W(TQ)\le\W(T),\qquad
 \E_{P,Q}\W(PTQ)\le\W(T).
\]
For example, conditional on $g\in\R^k$, the vector $P^Tg$ has law
$(\norm g_2/\sqrt{k})g'$ with $g'$ standard; use
$\E\norm g_2/\sqrt{k}\le1$. The three blocks in \eqref{eq:Phi} thus give
\begin{equation}\label{eq:Phi-width}
 \E_{P,Q}\W(\Phi(T))\le3\W(T).
\end{equation}
The expected number of bad indices is at most $0.01m$. By Markov's
inequality, there is a choice of $P,Q$ with at least $m/2$ good indices
and $\W(\Phi(T))\le12\W(T)$.

For this choice write, over good indices,
\[
 X_i=[\,PA_iQ\ \ PA_i\,],\qquad Y_i=(A_iQ)^T.
\]
These both have $k$ rows, and
\[
 \norm{\Phi(A_i)-\Phi(A_j)}_{\op}
 \le\norm{X_i-X_j}_{\op}+\norm{Y_i-Y_j}_{\op}.
\]
Take maximal $t$-separated sets in the two coordinate families, where
$t=c_0a/8$. Their $t$-balls cover the respective families. Two distinct
indices cannot be assigned to the same pair of balls, by
\eqref{eq:Phi-separation}. Therefore
\[
 m/2\le\Pack(\{X_i\},t)\Pack(\{Y_i\},t).
\]
Select a packing of cardinality at least $\sqrt{m/2}$ in one coordinate
family. Taking a block, and transposing when needed, contracts $\W$.
Equations \eqref{eq:Phi-radius} and \eqref{eq:Phi-width} prove all the
assertions.
\end{proof}

\section{Localization and completion of the endpoint proof}
\label{sec:localization}

The only previously established chaos minoration used as an input is
\begin{equation}\label{eq:known-HS}
 u\big(\log N(T,\norm{\cdot}_{\HS},u)\big)^{1/4}\le C S(T).
\end{equation}
See \cite[Proposition 15.1.15]{TalagrandBook} and \cite{TalagrandChaos}.

\begin{lemma}[Choice of a bounded scale]\label{lem:dyadic}
Suppose $T$ is an $a$-separated finite family with $p=\log|T|>0$.
There are $s\ge a$, $q>0$, and an $s$-separated subfamily $T_0\subset T$
such that
\begin{equation}\label{eq:dyadic-conclusion}
 s\sqrt q\ge a\sqrt p,\qquad
 \log|T_0|\ge3q/4,\qquad
 \operatorname{diam}_{\op}(T_0)\le4s.
\end{equation}
\end{lemma}
\begin{proof}
Set $q_j=\log\Pack(T,2^ja)$ for $j\ge0$ and choose $j$ maximizing
$4^jq_j$. The maximum exists since $q_j=0$ for all sufficiently large
$j$, and it is positive since $q_0=p$. Set $s=2^ja$, $q=q_j$.
Then $q_{j+1}\le q/4$ and $s^2q\ge a^2p$.

A maximal $2s$-separated subset covers $T$ by at most $e^{q/4}$ operator
balls of radius $2s$. An $s$-packing with $e^q$ points has at least
$e^{3q/4}$ points in one such ball. Their diameter is at most $4s$.
\end{proof}

\begin{proof}[Proof of Theorem~\ref{thm:main}]
The assertion is trivial for $m=1$. Put $M=S(T)$ and apply
Lemma~\ref{lem:dyadic}, obtaining $s,q,T_0$. Since $q>0$, the original
family contains two points separated by at least $s$, so
\begin{equation}\label{eq:x-lower}
 M\ge s/\pi.
\end{equation}
If $q\le16$, \eqref{eq:dyadic-conclusion} and \eqref{eq:x-lower} already
prove the desired result. Hence assume $q>16$.

Choose $u=C_0M/q^{1/4}$, with $C_0$ large enough that
\eqref{eq:known-HS}, applied to $T_0$, supplies a Hilbert--Schmidt cover
with at most $e^{q/8}$ balls of radius $u$. One ball contains at least
$e^{5q/8}$ points. Translate those points by one of them and call the
resulting family $T_1$. Then
\begin{align}
 \log|T_1|&\ge5q/8,
 &\norm{A-B}_{\op}&\ge s\quad(A\ne B),\notag\\
 \sup_{A\in T_1}\norm A_{\op}&\le4s,
 &\sup_{A\in T_1}\norm A_{\HS}&\le CM/q^{1/4}.
 \label{eq:localized-family}
\end{align}
Translation preserves $S$, and $0\in T_1$, so $\W(T_1)\le2M$.

Apply Lemma~\ref{lem:compression} to $T_1$. We obtain a family
$F\subset\R^{r\times N}$ with
\begin{align}
 \log|F|&\ge q/4,\label{eq:final-family-card}\\
 \norm{B-B'}_{\op}&\ge cs\quad(B\ne B'),
 &\sup_{B\in F}\norm B_{\op}&\le Cs,\notag\\
 \W(F)&\le CM,
 &r&\le C\left(1+\frac{M^2}{s^2\sqrt q}\right).
 \label{eq:final-family-bounds}
\end{align}
The constants in the radius-to-separation ratio are numerical, so
Corollary~\ref{cor:bounded-row} applies with a fixed numerical $L$.
Write $x=M/s$ and $y=\sqrt q$. Equations
\eqref{eq:final-family-card}--\eqref{eq:final-family-bounds} give
\begin{equation}\label{eq:closing-quadratic}
 y\le C(1+r+x)\le C(1+x+x^2/y).
\end{equation}
Multiplication by $y$ and the quadratic formula imply
$y\le C'(1+x)$. Since $x\ge1/\pi$ by \eqref{eq:x-lower}, it follows
that $s\sqrt q\le CM$. Finally,
\[
 a\sqrt{\log|T|}\le s\sqrt q\le CS(T).
\]
This proves \eqref{eq:main-packing}. A maximal
$\varepsilon$-separated subset is an $\varepsilon$-cover, so applying
the packing assertion to that subset proves \eqref{eq:main-covering}.
\end{proof}

\begin{remark}[Why the order of reductions matters]
The dyadic selection need not retain a fixed fraction of the original
entropy. It retains the quantity $s\sqrt q$, which is the numerator of
the desired Sudakov ratio. The Hilbert--Schmidt radius and the common row
dimension are then recomputed using this new $q$. This gives simultaneously
a fixed radius-to-separation ratio and
$r\lesssim1+(M/s)^2/\sqrt q$. Keeping a rank bound based on the old entropy
would not justify \eqref{eq:closing-quadratic}.
\end{remark}

\begin{corollary}[Equal-rank projections]\label{cor:projections}
Let $P_1,\ldots,P_m$ be orthogonal projections of the same rank, satisfying
$\norm{P_i-P_j}_{\op}\ge a$ for $i\ne j$. For a standard Gaussian vector
$z$,
\begin{equation}\label{eq:projection-oscillation}
 \E\left[\max_i z^TP_i z-\min_i z^TP_i z\right]
 \ge c a\sqrt{\log m}.
\end{equation}
\end{corollary}
\begin{proof}
Let $g,h$ be independent standard Gaussian vectors and put
$u=(g+h)/\sqrt2$, $v=(g-h)/\sqrt2$. These are independent standard
Gaussian vectors, and symmetry of the projections gives
\[
 g^TP_i h=\tfrac12(u^TP_i u-v^TP_i v).
\]
Taking maxima and expectations shows that twice the bilinear supremum is
at most the oscillation in \eqref{eq:projection-oscillation}.
Apply Theorem~\ref{thm:main}.
\end{proof}

\begin{remark}[Hilbert-space matrices]
The theorem also applies to a finite family of Hilbert--Schmidt operators
between separable real Hilbert spaces. Compress onto increasing finite
dimensional subspaces. Hilbert--Schmidt convergence implies operator-norm
convergence and $L^2$ convergence of each associated chaos. For a finite
family the expected supremum therefore converges. Passing to the limit
proves the same inequality. Arbitrary index sets are treated through their
finite separated subfamilies whenever the supremum is well defined.
\end{remark}

\end{document}